\documentclass[12pt]{amsart}
\usepackage{amsfonts,color,latexsym,amssymb,amsmath,epsfig,rotating,array}

\usepackage{url} 

\theoremstyle{definition}
\newtheorem{prp}{Proposition}[section]
\newtheorem{thm}[prp]{Theorem}
\newtheorem{lem}[prp]{Lemma}
\newtheorem{dfn}[prp]{Definition}

\newtheorem{rem}[prp]{Remark}

\newtheorem{algor}[prp]{Algorithm}
\newtheorem{nota}[prp]{Notation}

\newcommand{\C}{\mathbb{C}}
\newcommand{\F}{\mathbb{F}}
\newcommand{\N}{\mathbb{N}}
\newcommand{\Q}{\mathbb{Q}}

\newcommand{\Z}{\mathbb{Z}}

\newcommand{\cC}{\mathcal{C}}

\newcommand{\cT}{\mathcal{T}}

\newcommand{\dia}{$\diamond$}

\newcommand{\modi}[1]{\!\left/\!\left\langle#1\right\rangle\right.}
\newcommand{\ord}{\operatorname{ord}}

\newcommand{\tf}{\tilde{f}}
\newcommand{\zp}{\Z/p\Z}
\newcommand{\zpt}{(\Z/p\Z)^2}
\newcommand{\zpk}{\Z/p^k\Z}
\newcommand{\zpkt}{(\Z/p^k\Z)^2}
\newcommand{\tpk}{T_{p,k}}
\newcommand{\ctpk}{\cT_{p,k}}

\newcommand{\bx}{\mathbf{x}}
\newcommand{\bt}{\mathbf{t}}
\newcommand{\bep}{\varepsilon}
\newcommand{\mzf}{m_{\zeta}(\tf)}

\newcommand{\floor}[1]{\left\lfloor #1 \right\rfloor}
\newcommand{\ceil}[1]{\left\lceil #1 \right\rceil}

\begin{document}
\title[Sub-Linear Point Counting for Curves over Prime Power 
Rings]{\mbox{}\\ 
\vspace{-.75in}
Sub-Linear Point Counting for Variable Separated Curves over Prime Power Rings}

\author{Caleb Robelle}
\thanks{C.B. was partially supported by NSF grant DMS-1757872.}

\author{J.\ Maurice Rojas}
\thanks{J.M.R. and Y.Z. were partially supported by NSF grants CCF-1900881 and 
DMS-1757872.} 

\author{Yuyu Zhu}

\maketitle

\mbox{}\\
\vspace{-.7in}
\begin{abstract}
Let $k,p\in \N$ with $p$ prime and let $f\!\in\!\Z[x_1,x_2]$ be a bivariate 
polynomial with degree $d$ and all coefficients of absolute value at most 
$p^k$. Suppose also that $f$ is variable separated, i.e., 
$f\!=\!g_1+g_2$ for $g_i\!\in\!\Z[x_i]$. We 
give the first algorithm, with complexity sub-linear in $p$, to count the 
number of roots of $f$ over $\Z\modi{p^k}$ for arbitrary $k$: 
Our Las Vegas randomized 
algorithm works in time $(dk\log p)^{O(1)}\sqrt{p}$, and admits a 
quantum version for smooth curves working in time $(d\log p)^{O(1)}k$. Save 
for some subtleties concerning non-isolated singularities, our techniques 
generalize to counting roots of polynomials 
in $\Z[x_1,\ldots,x_n]$ over $\Z\modi{p^k}$. 

Our techniques are a first step toward efficient point counting 
for varieties over Galois rings (which is relevant to error 
correcting codes over higher-dimensional varieties), and also imply 
new speed-ups for computing Igusa zeta functions of curves. The latter 
zeta functions are fundamental in arithmetic geometry.  
\end{abstract}

\bigskip \bigskip
\noindent
{\sc Current affiliation and address of authors:}\\

\noindent
(Robelle):\\  
University of Maryland, Baltimore County\\ 
1000 Hilltop Circle\\ 
Baltimore, MD \ 21250 

\bigskip 

\noindent
(Rojas \& Zhu):\\  
Texas A\&{}M University,
Department of Mathematics\\ 
TAMU 3368\\ 
College Station, TX \ 77845

\bigskip \bigskip
\noindent
{\sc emails: } {\tt carobel1@umbc.edu} , 
 {\tt rojas@math.tamu.edu} , 
{\tt zhuyuyu@math.tamu.edu}

\thispagestyle{empty} 

\newpage 

\addtocounter{page}{-1} 

\section{Introduction} 
Counting points on algebraic curves over finite fields is a seemingly 
simple problem that nevertheless helped form the core of arithmetic geometry 
in the 20th century and now forms an important part of cryptography 
\cite{miller,koblitz,galbraith} and coding theory \cite{vandergeer}. Efficient 
algorithms for this problem 
continue to be a lively part of computational number theory: The barest list 
of references would have to include 
\cite{sch85,pil90,ah01,ked01,cdv06,lauderwan,dwan,
chambert,har15}.\footnote{Also, major conferences such as ANTS consistently 
continue to feature papers on speeding up point-counting for various special 
families of curves and surfaces.} Here, we consider algorithms for the natural 
extension of this problem to prime power rings, and find the first efficient 
algorithms for a broad class of (not necessarily smooth) curves:  
See Theorem \ref{thm:sepcurve} below. It will be useful to first discuss 
some motivation before covering further background. 

\subsection{A Connection to Error Correcting Codes} 
Suppose $k,p\!\in\!\N$ with $p$ prime, $\F_p$ is the field with $p$ elements, 
and $r\!\in\!\Z[x_1]$ is a univariate polynomial of degree 
$m$ that is irreducible mod $p$. We call a quotient ring $R$ of the form 
$\Z[x_1]\modi{p^k,r(x_1)}$ a {\em Galois ring}. Note that such an $R$ is 
finite, and can be 
the prime power ring $\Z\modi{p^k}$ (for $m\!=\!1$) or the 
field $\F_q$ (for $k\!=\!1$ and $q\!=\!p^m$), to name a few examples. 

Since numerous error correcting 
codes and cryptosystems are based on arithmetic over $\F_q$ or $\F_q[x_1]$, 
it has been observed (see, e.g., \cite{mullen91,gusasu,blq13,cohnheninger}) 
that one can generalize and improve these constructions by using arithmetic 
over $R$ or $R[x_1]$ instead. For instance, Guruswami and Sudan's famous 
list-decoding method for error correcting codes \cite{gs99} involves finding 
the roots in $\F_q[x_1]$ of a polynomial in $\F_q[x_1,x_2]$ as a key step, 
and has a natural generalization to Galois rings (see, e.g.,   
\cite{hammons,sudan97,wasan} and \cite[Sec.\ 4]{blq13}). 
Furthermore, counting solutions to equations like $f(x_1,\ldots,x_n)\!=\!0$ 
over Galois rings determines the weights of codewords in Reed-Muller codes 
over Galois rings, and the weight distribution governs the quality of the 
underlying code (see, e.g., \cite{lovett}). 

\subsection{Connections to Zeta Functions and Rational Points} 
Efficiently counting roots in 
$\left(\Z\modi{p^k}\right)^2$ of polynomials in $\Z[x_1,x_2]$ 
is a natural first step toward efficiently enumerating the roots in $R^2$ for 
polynomials in $R[x_1,x_2]$ for $R$ a Galois ring. However, observe that  
the ring of $p$-adic integers $\Z_p$ is the inverse limit of 
$\Z\modi{p^k}$ as $k\longrightarrow \infty$. It then turns out that the zero 
sets of polynomials over $\Z\modi{p^k}$ inform the zero sets of polynomials 
over $\Z_p$ and beyond. 

In particular, for any $f\!\in\!\Z[x_1,\ldots,x_n]$, one can 
form a fundamentally important generating function, and a 
related zeta function, as follows:  
Let $N_{p,k}(f)$ denote the number of 
roots in $\left(\Z\modi{p^k}\right)^n$ of the mod $p^k$ reduction of $f$ 
and define the {\em Poincare series of $f$} to be 
$P_f(t)\:=\!\sum_{k=0}^{\infty}\frac{N_{p,k}(f)}{p^{kn}}t^k$.  
Also, letting $t\!:=\!p^{-s}$, we define the 
{\em Igusa local zeta function of $f$} to be $Z_f(t)\!:=\!\int_{\Z_p}
|f(x_1,\ldots,x_n)|^s_pdx$, where $|\cdot|_p$ and $dx$ respectively 
denote the standard $p$-adic absolute value on $\Z_p$ and Haar 
measure on $\Z_p$. (This function turns to be defined on the right open 
half-plane of $\C$, possibly with the exception of finitely many poles.) 
The precise definitions of $|\cdot|_p$ and $dx$  
won't matter for our algorithmic results, but what does matter is that Igusa 
discovered in the 1970s that $P(t)\!=\!\frac{1-tZ(t)}{1-t}$ and 
proved that $Z$ (and thus $P$) is a rational function of $t$ \cite{igu}. 

Igusa defined his zeta function $Z$ with the goal of generalizing 
earlier work of Siegel (on counting representations of integers via quadratic 
forms) to high degree forms, e.g., 
how many ways can one write $239$ as a sum of cubes? However, 
the algorithmic computation of these zeta functions has received 
little attention, aside from some very specific cases. Our 
results imply that one can compute $Z$ for certain bivariate 
$f$ in time polynomial in $dk\log p$. This extends earlier work on 
the univariate case \cite{dwivedi,zhu} to higher-dimensions and will 
be pursued in a sequel to this paper. 

It should also be pointed out that recent algorithmic methods for 
finding rational points (over $\Q$) for curves of genus $\geq\!2$ 
proceed (among many other difficult steps) by finding the $p$-adic 
rational points on a related family of varieties (see, e.g., 
\cite[Sec.\ 5.3]{bala}). So a long term 
goal of this work is to improve the complexity of 
finding the $p$-adic rational points on curves and surfaces, 
generalizing recent $p$-adic speed-ups in the univariate case 
\cite{rojaszhu}.

\subsection{From Finite Fields to Prime Power Rings} 
Returning to point counting over prime power rings, the computation 
of $N_{p,k}(f)$ is subtle already for $n\!=\!1$: 
This special case has recently been addressed from different perspectives 
in \cite{blq13,cgrw2,krrz19,DMS19}, and was just recently proved to 
admit a {\em deterministic} algorithm of complexity $(dk\log p)^{O(1)}$, thanks 
to the last paper.  

The special case $(n,k)\!=\!(2,1)$ of computing $N_{p,1}(f)$, just 
for $f$ a cubic polynomial, is already of 
considerable interest in the design of cryptosystems based on the elliptic 
curve discrete logarithmic problem. In fact, even this very special case 
wasn't known to admit an algorithm polynomial in $\log p$ until Schoof's work 
in the 1980s \cite{sch85}. 
More recently, algorithms for computing 
$N_{p,1}(f)$ for arbitrary $f\!\in\!\Z[x_1,x_2]$ of degree $d$, with complexity 
$d^8(\log p)^{2+o(1)}\sqrt{p}$, have been derived by Harvey \cite{har15} 
(see also \cite[Ch.\ 5]{zhu}),  
and similar complexity bounds hold for arbitrary finite fields.  

Our main result shows that counting points over $\Z\modi{p^k}$ for arbitrary 
$k$ is slower than the $k\!=\!1$ case only by a factor polynomial in $k$ 
(neglecting the other parameters). 
\begin{thm} {\em  
\label{thm:sepcurve}
Suppose $f\!=\!g_1+g_2$ for some $g_i\!\in\!\Z[x_i]$, 
$\deg f\!=\!d\!\geq\!1$, and all the coefficients of $f$ are of absolute 
value at most $p^k$. Then there is a Las Vegas randomized algorithm that 
computes $N_{p,k}(f)$ in time $d^{17+\varepsilon}(k\log p)^{2+\varepsilon}p^{1/2+\varepsilon}$. In particular, 
the number of random bits needed is $O(d^2k\log(dk)\log p)$, and the space 
needed is $O(d^4k\sqrt{p}\log p)$. Furthermore, if the zero set of 
$f$ over the algebraic closure $\bar{\F}_p$ is smooth and irreducible, then $N_{p,k}(f)$ can be computed in 
quantum randomized time $(d(\log p))^{O(1)}k$. } 
\end{thm}

\noindent 
We prove Theorem \ref{thm:sepcurve} in Section \ref{sec:algo}. The central 
idea is to reduce to a moderate number of moderately sized instances of point 
counting over $\F_p$. Recall that {\em Las Vegas randomized time} simply 
means that our algorithm needs random bits and gives an answer that is 
correct with probability at least $1/2$ and, in case of error, states 
that an error has occured. {\em Quantum} randomized time here will mean 
that we avail to a quantum computer, and instead obtain an algorithm 
that gives an answer that is correct with probability at least $2/3$, 
but with no correctness guarantee. 

In what follows, we call a polynomial of the form 
$f_\zeta(x_1,x_2)\!:=\!\frac{1}{p^s}f(\zeta_1+px_1,\zeta_2+px_2)$, with 
$(\zeta_1,\zeta_2)\!\in\!\F^2_p$ a singular point of the zero set of 
$f$ in $\F^2_p$ and $s$ as large as possible with $f_\zeta$ still in 
$\Z[x_1,x_2]$, a {\em perturbation} of $f$. Our reduction to point counting 
over $\F_p$ will involve finding all {\em isolated} singular points of 
the zero set of $f$ (as well as its perturbations) in $\F^2_p$, in order to 
categorize the base-$p$ digits of the coordinates of the roots of $f$ in 
$\left(\Z\modi{p^k}\right)^2$. This yields a geometrically defined 
recurrence for $N_{p,k}(f)$ that is conveniently encoded by a tree. We 
detail this construction in Sections \ref{sub:rec} and \ref{sec:algo} below. 

\begin{rem} 
{\em A classical algebraic geometer may propose simply applying resolution 
of singularities, applying finite field point counting (with proper corrections 
at blown-up singular points), and then an application of Hensel's Lemma. We use 
a more direct approach that allows us to lift singular points individually  
and much more simply. In particular, it appears (from \cite{poteaux})  
that resolution of singularities for a plane curve of degree $d$ 
over $\F_p$ has complexity $O(d^5)$ (neglecting multiples depending  
on $p$), while our algorithm (if looked at more closely) has better 
dependence on $d$. More to the point, replacing an input bivariate  
polynomial by a higher degree complete intersection (the latter being the 
output after doing resolution of singularities) results in a more 
complicated input when one needs to avail to prime field point counting, 
thus compounding the complexity even further. Furthermore, in higher 
dimensions, resolution of singularities becomes completely impractical 
\cite{grigoriev}. 
\dia } 
\end{rem} 
\begin{rem}{\em 
We can extend Theorem \ref{thm:sepcurve} to more general 
curves. The key obstruction is whether $f$, or one of its perturbations, 
fails to be square-free (see the final section of the Appendix). 
We hope to extend our methods to arbitrary 
curves in the near future. For now, we simply point out that many 
commonly used curves in practice are variable separated, e.g., 
many hyperelliptic curves used in current cryptography are zero sets 
of polynomials of the form $x^2_2-g(x_1)$. \dia 
}
\end{rem}

\section{Background} 
\label{sec:back} 
\subsection{Some Basics on Point Counting Over Finite Fields} 
One of the most fundamental results on point counting for 
curves over finite fields dates back to work of Hasse and Weil 
in the 1940s. In what follows, we use $|S|$ to denote the cardinality 
of a set $S$. 
\begin{thm} \cite{wei49} 
{\em Let $\F_q$ be a finite field of order $q = p^m$, and let $\cC$ be an 
absolutely irreducible smooth projective curve defined over $\F_q$. Let $g$ 
denote the genus of $\cC$ and $\cC(\F_q)$ to be the set of $\F_q$-points of 
$\cC$. Then $||\cC(\F_q)| - q | \leq 2g\sqrt{q}$.}
\end{thm}

\noindent
The error bound above is optimal, and can be derived by proving 
a set of technical statements known as the {\em Weil Conjectures (for curves)}. 
The Weil Conjectures (along with corresponding point counts) were formulated 
for arbitrary varieties over 
finite fields and, in one of the crowning achievements of 20th century 
mathematics, were ultimately proved by Deligne in 1974 \cite{del74}.

Efficient methods for computing $N_{p,1}(f)$ (and the number of 
points for a curve over any finite field) began to appear with the 
work of Schoof \cite{sch85}, via so-called $\ell$-adic methods.  
Let $g$ denote the genus\footnote{The precise definition of 
the genus need not concern us, so we will simply recall that it is a birational 
invariant of $\cC$ (i.e., it is invariant under rational maps with rational 
inverse) and is at most $(d-1)(d-2)/2$ for $\cC$ the zero set of a degree 
$d$ bivariate polynomial.} of the curve $\cC$. Via later work (e.g., 
\cite{pil90,ah01}) it was determined that $N_{p,1}(f)$ can be computed in time 
$(\log p)^{2^{g^{O(1)}}}$ for arbitrary curves. Kedlaya's algorithm 
\cite{ked01} then lowered 
this complexity bound to $(g^4p)^{1+o(1)}$ for hyperelliptic curves, 
e.g., curves with defining polynomials of the form $x^2_2-g(x_1)$.  
Kedlaya observed later that, on a quantum computer, one could 
compute (finite field) zeta functions for non-singular curves in 
time $(d\log p)^{O(1)}$ \cite{kedlaya}. 
(The precise definition of these zeta functions 
need not concern us here: Suffice it to say that the computation 
of the zeta function of a curve over a finite field includes the computation 
of $N_{p,1}(f)$ as a special case.) 
More recently, Harvey \cite{har15} gave 
an efficient (classical) deterministic  
algorithm which, although asymptotically slower than Kedlaya's quantum 
algorithm, allows arbitrary input polynomials.

\subsection{The Central Recurrence for Bivariate Point Counting}
\label{sub:rec} 
In this section, we generalize the tools we used for root counting for 
univariate polynomials in \cite{krrz19} to point counting for curves. It is 
not hard to see that these tools extend naturally to point counting for 
hypersurfaces of arbitrary dimension. The only subtlety is maintaining 
low computational complexity and keeping track of the underlying 
singular locus.  

Let $\bx:=(x_1,x_2)$ denote the tuple of two variables, and let $f(\bx)\in \Z[\bx]$ be a bivariate polynomial with integer coefficients of total degree $d\geq 1$. Then for any $\zeta:=(\zeta_1, \zeta_2)\in \Z^2$, the Taylor expansion of 
$f$ at $\zeta$ is
$f(\bx) = \sum_{j_1,j_2}\frac{D^{j_1,j_2}f(\zeta)}{j_1!j_2!}(x_1-\zeta_1)^{j_1}
(x_2-\zeta_2)^{j_2}$,
where $j_1,j_2$ are non-negative integers and $D^{j_1,j_2}f(\bx):=\frac{\partial^{j_1+j_2}}{\partial x_1^{j_1}\partial x_2^{j_2}}f(\bx)$. 

Let $\tf(\bx):=(f(\bx)\mod p)$ denotes the mod $p$ reduction of $f$. 
Now let $\zeta = (0,0)$ and write $\tf = g_{m}+g_{m+1}+\cdots+g_n$ where 
$g_i$ is a (homogeneous) form in $\F_p[\bx]$ of degree $i$ and $g_m\neq 0$. We 
then define $m$ to be the {\it multiplicity of $\tf$ at $\zeta = (0,0)$}. Write $m = m_{\zeta}(\tf)$. To extend this definition to a point $\zeta = (a,b)\neq (0,0)$, let $T$ be the translation that takes $(0,0)$ to $\zeta$, i.e. $T(x_1,x_2) = (x_1+a, x_2+b)$. Then $\tf^T:=\tf(x_1+a, x_2+b)$  and we define $m_{\zeta}(\tf):=m_{(0,0)}(\tf^T)$. Then it is immediate from the definition that:
\begin{lem}\label{lem:func}{\em
	If $\tf = \prod\tf_{r}^{e_r}\in \F_p[\bx]$ is a factorization of 
$\tf$ into irreducible polynomials over $\F_p$ then 
$\mzf = \sum m_{\zeta}(\tf_{r})$.}
\end{lem}

\noindent
We say $\zeta$ is a {\it smooth point} of $\tf$ if $m_{\zeta}(\tf) = 1$, and call it a {\it singular point} otherwise. In particular, by Lemma \ref{lem:func}, a point $\zeta$ is a smooth point of $\tf$ if and only if $\zeta$ belongs to just one irreducible component $\tf_r$ of $\tf$, the corresponding exponent $e_r = 1$, and $\zeta$ is a smooth point of $\tf_r$.

Now we are ready to generalize the tools in \cite{krrz19} for curves:
\begin{dfn} \label{dfn:ntpk}
{\em Let $f(\bx)\!\in\!\Z[\bx]$ and fix a prime $p$. Let $\ord_p: \Z 
\longrightarrow \N\cup\{0\}$ denote the usual $p$-adic valuation with 
$\ord_p(p)\!=\!1$. We then define $s(f, \bep):=\min_{j_1,j_2\geq 0}
\left\lbrace j_1+j_2+\ord_p\frac{D^{j_1,j_2}f(\bep)}{j_1!j_2!}\right\rbrace$ 
for any $\bep\in \{0,\ldots,p-1\}^2$. Finally, fixing $k\!\in\!\N$, let us 
inductively define a set $T_{p,k}(f)$ of pairs 
$(f_{i,\zeta},k_{i,\zeta})\!\in\!\Z[\bx]\times \N$ as follows: We set 
$(f_{0,0},k_{0,0})\!:=\!(f,k)$. Then, for any $i\!\geq\!1$ with 
$(f_{i-1,\mu},k_{i-1,\mu})\!\in\!T_{p,k}(f)$ and any {\em singular} point 
$\zeta_{i-1}\!\in \zpt$ of $\tilde{f}_{i-1,\mu}$ with 
$s_{i-1}\!:=\!s(f_{i-1,\mu},\zeta_{i-1})\!\in\!\{2,\ldots,k_{i-1,\mu}-1\}$, we 
define $\zeta:=\mu+p^{i-1}\zeta_{i-1}$, $k_{i,\zeta}:=k_{i-1,\mu}
-s_{i-1}$ and $f_{i,\zeta}(\bx):=                     
\left[\frac{1}{p^{s_{i-1}}}f_{i-1,\mu}(\zeta_{i-1}+p\bx)\right] \ \mathrm{mod} 
\ p^{k_{i,\zeta}}$. }
\end{dfn} 
Just as in the univariate case, the perturbations $f_{i,\zeta}$ of $f$ will 
help us keep track of how the points of $f$ in $\zpkt$ cluster, in a 
$p$-adic metric sene, about the points of $\tf$. It is clear that 
$\frac{D^{j_1,j_2}f(\bep)}{j_1!j_2!}$ is always an integer as the coefficient 
of $x_1^{j_1}x_2^{j_2}$ in the Taylor expansion of $f(\bx+\bep)$ about 
$\bx = (0,0)$. We will see in the next section how $\tpk(f)$ is associated 
with a natural tree structure. Moreover, $\tpk(f)$ is always a finite set by 
definition, as only $f_{i,\zeta}$ with $i\leq \floor{(k-1)/2}$ and 
$\zeta\in \zpt$ are possible. 

\begin{lem}\label{lem:nrec}{\em
Following the notation above, let $n_p(f)$ denote the number of smooth points 
of $\tf$ in $\zpt$. Then provided $k\geq 0$ and $\tf$ is not identically zero, we have  \[N_{p,k}(f) \ = \ p^{k-1}n_p(f) 
	+ \left(\sum\limits_{\substack{\zeta_0\in(\Z/p\Z)^2\\ 
			s(f,\zeta_0)\geq k}} p^{2(k-1)}\right) 
	+ \!\!\!\!\!\!\!  
	\sum\limits_{\substack{\zeta_0\in(\Z/p\Z)^2\\
			s(f,\zeta_0)\in\{2,\ldots,k-1\}}} 
	\!\!\!\!\!\!\! 
	p^{2(s(f,\zeta_0)-1)}N_{p,k-s(f,\zeta_0)}(f_{1,\zeta_0})
	.\]  }
\end{lem}
We will prove Lemma \ref{lem:nrec} in the next section, where it will be clear how Lemma \ref{lem:nrec} applies recursively. Then we show how Lemma \ref{lem:nrec} leads to our recursive algorithm for computing $N_{p,k}(f)$.

\section{Generalized Hensel Lifting and the Proof of our Main Recurrence}
Let us first prove the following alternative definition for multiplicity of a point on the curve. We will mainly use this definition for the rest of the discussion.

\begin{lem}\label{lem:alt}
{\em For any $\zeta\in \F_p^2$, $m:=\mzf$ is the smallest nonnegative integer such that there exists $j_1, j_2\geq 0$ with $j_1+j_2 = m$, and $D^{j_1,j_2}f(\zeta)\neq 0\mod p$. }
\end{lem}
\begin{proof}
Fix $\zeta\in \F_p^2$, and let $T$ be the translation that takes $(0,0)$ to $\zeta$. Then for any $j_1, j_2\geq 0$, $D^{j_1,j_2}\tf^T(0,0) = D^{j_1,j_2}\tf(\zeta)$. So it suffices to prove the statement for the case when $\zeta = (0,0)$. 

Suppose $\tf = g_m + g_{m+1} + \cdots + g_n$, where $g_i$ is a homogeneous form in $\F_p[\bx]$ of degree $i$ and $g_m\neq 0$. Then $\tf$ must have a nonzero monomial term $a_rx_1^rx_2^{m-r}$, for some integer $r\leq m$, and $a_r\in \F_p^{\times}$. Note that as $h_m \in \F_p[\bx]$, we must have $r, m-r < p$ as well. 
Then for any $j_1, j_2\geq 0$, we have 
$D^{j_1,j_2}\left( a_rx_1^rx_2^{m-r} \right) = a_r{r\choose r-j_1}{m-r \choose m-r-j_2}x_1^{r-j_1}x_2^{m-r-j_2}$. 
It is obvious that for any pair of nonnegative integers $j_1,j_2$ with $j_1+j_2<m$, either $r-j_1 > 0$ or $m-r-j_2>0$. Moreover, any other nonzero monomial term $a_tx_1^{t_1}x_2^{t_2}$ of $\tf$ must have $t_1 + t_2 \geq m$ and $t_1\geq r$ or $t_2\geq m-r$. Hence $t_1-j_1 > 0$ or $t_2-j_2 > 0$. So for such a pair of $j_1, j_2$, we must have $D^{j_1,j_2}\tf(0,0) = 0\mod p$. Now take $j_1 = r$ and $j_2 = m-r$, then 
\begin{align*}
D^{j_1,j_2}\tf(0,0) = a_r{r\choose r-j_1}{m-r \choose m-r-j_2}\neq 0 \mod p.
\end{align*}

Conversely, if $m$ is the smallest nonnegative integer such that there exists $j_1,j_2\geq 0$ with $j_1+j_2 = m$ and $D^{j_1,j_2}f(0,0)\neq 0\mod p$, then there exists $a_jx_1^{j_1}x_2^{j_2}$ a nonzero monomial term of $\tf$ of smallest total degree. So $m = m_{(0,0)}(\tf)$.
\end{proof}

The classical Hensel's Lemma (see, e.g., \cite[Thm.\ 2.3, Pg.\ 87]{nzm91}) says 
that any {\it non-degenerate} root of a univariate polynomial in $\zp$ lifts 
uniquely into any larger prime power ring $\zpk$. One expects similar nice 
behavior from a smooth point on a curve over $\zp$. We prove the following 
analogue of Hensel's Lemma for curves in the Appendix:

\begin{lem}\label{lem:nhensel}
{\em Let $f(\bx)\in \Z[\bx]$. If $f(\sigma) \equiv 0 \mod p^j$ for $j\geq 1$, 
and $\left( \zeta^{(0)} \equiv \sigma \mod p\right) $ is a smooth point on 
$\tf$, then there are exactly $p$ many $\bt\in (\Z/p\Z)^2$ such that 
$f(\sigma+p^j\bt) \equiv 0\mod p^{j+1}$.}
\end{lem}

For $k>j\geq 1$ and any $\sigma^{(j)}\in (\Z/p^j\Z)^2$ such that $f(\sigma^{(j)})\equiv 0\mod p^j$, we call $\sigma^{(k)}\in (\Z/p^k\Z)^2$ a {\it lift} of $\sigma^{(j)}$, if $f(\sigma^{(k)})\equiv 0\mod p^k$ and $\sigma^{(k)} \equiv \sigma^{(j)}\mod p^j$. Then by applying Lemma \ref{lem:nhensel} inductively, we obtain:
\begin{prp} \label{prp:smooth}
	{\em 
	Let $f(\bx)\in \Z[\bx]$, and $k> j\geq 1$. If $f(\sigma^{(j)})\equiv 0\mod p^j$, and $(\sigma^{(j)}\mod p)$ is a smooth point of $\tf$, then $\sigma^{(j)}$ lifts to exactly $p^{k-j}$ many roots of $(f\mod p^k)$.}
\end{prp}

\begin{lem} \label{lem:nmul}
{\em Following the notation above, suppose instead $\zeta^{(0)}\in (\zp)^2$ is a 
point on $\tf$ of (finite) multiplicity $m \geq 2$. Suppose also that $k\geq 2$ 
and that there is a $\sigma^{(k)}\in (\zpk)^2$ with $\sigma^{(k)}\equiv \zeta^{(0)}\mod p$ and $f(\sigma^{(k)}) =0\mod p^k$. Then $s(f,\zeta^{(0)})\in \{2,\ldots,m\}$.}
\end{lem}
\begin{proof}
	As $\zeta^{(0)}$ is a singular point on $\tf$, then $\frac{\partial f}{\partial x_i}(\zeta^{(0)})=0\mod p$ for every $i=1,\ldots,n$. Then for $\sigma^{(k)} = \zeta^{(0)}+p\tau \in (\zpk)^2$ with $\tau:=(\tau_1,\tau_2)\in (\Z/p^{k-1}\Z)^2$,
	{\small \begin{align}\label{eq:mul}
		f(\sigma^{(k)}) = f(\zeta^{(0)}) + p\left( \frac{\partial f}{\partial x_1}(\zeta^{(0)})\tau_1 + \frac{\partial f}{\partial x_2}(\zeta^{(0)})\tau_2\right) + \sum_{i_1+i_2\geq 2}p^{i_1+i_2} D^{i_1+i_2}f(\zeta^{(0)})\tau_1^{i_1} \tau_2^{i_2}
		\end{align}}
	to have solutions mod $p^k$, we need $f(\zeta^{(0)}) \equiv 0\mod p^2$, as the second and the third summand in equation (\ref{eq:mul}) has $p$-adic order at least $2$. 
	
	As $\zeta^{(0)}$ is a singular point of multiplicity $m$ on $\tf$, there exists an $m$-th Hasse derivative: $D^{j_1, j_2}f(\zeta^{(0)})\neq 0\mod p$ with $j_1+j_2=m$. So $s(f,\zeta^{(0)}) \leq \ord_p\left( p^{j_1+j_2}D^{j_1,j_2}f(\zeta^{(0)})\right) =m$.
\end{proof}

We can now relate $N_{p,k}(f)$ to the recursive structure on $T_{p,k}(f)$.

\medskip 
\noindent 
{\bf Proof of Lemma \ref{lem:nrec}:} 
The lifting of smooth points of $\tf$ follows from Proposition \ref{prp:smooth}.

Now assume that $\zeta_0\in (\zp)^2$ is a singular point of $\tf$. Write $\zeta := \zeta_0+p\sigma$ for $\sigma:=\zeta_1+p\zeta_2+\cdots+p^{k-2}\zeta_{k-1} \in 
(\zpk)^2$, and let $s:=s(f,\zeta_0)$. Note that by Lemma \ref{lem:nmul}, 
$s\geq 2$. Then by definition, $f(\zeta) = p^sf_{1,\zeta_0}(\sigma)$, for 
$f_{1,\zeta_0}\in \Z[\bx]$ and $f_{1,\zeta_0}$ does not vanish identically mod 
$p$.

If $s\geq k$, then $f(\zeta)=0\mod p^k$ regardless of choice of $\sigma$. So there are exactly $p^{2(k-1)}$ values of $\zeta\in (\zpk)^2$ such that $\zeta \equiv \zeta_0\mod p$ and $f(\zeta) = 0\mod p^k$.

If $s\leq k-1$ then $\zeta$ is a root of $f$ if and only if 
$f_{1,\zeta_0} (\sigma) \equiv 0 \mod p^{k-s}$. 
But then $\sigma = \zeta_1+p\zeta_2+\ldots+p^{k-s-1}\zeta_{k-s} \mod p^{k-s}$, 
i.e., the rest of the base $p$ digits $\zeta_{k-s+1},\ldots,\zeta_{k-1}$ do not 
appear in the preceding mod $p^{k-s}$ congruence. So the number of possible lifts 
$\zeta$ of $\zeta_0$ is exactly $p^{2(s-1)}$ times the number of roots 
$(\zeta_1+p\zeta_2+\ldots+p^{k-s-1}\zeta_{k-s}) \in (\Z/p^{k-s}\Z)^2$ of 
$f_{1,\zeta_0}$. This accounts for the third summand in our formula.  \qed

\begin{rem}{\em
The algebraic preliminaries we concluded in this section and Definition \ref{dfn:ntpk} can be extended transparently for point counting for hypersurfaces of arbitrary dimensions. \dia }
\end{rem}

\section{Bounding Sums of Multiplicities on Curves with at Worst Isolated Singularities}
Suppose $F \in \F_p[\bx]$ is a nonconstant polynomial of total degree $D$. Then $F$ factors into a product of irreducible components $F = \prod_{i=1}^{l} F_i^{e_i}\in \F_p[\bx]$ where each $F_i\in \F_p[\bx]$ is irreducible, and $e_i\geq 1$. We say $F$ is {\em squarefree} if $e_i = 1$ for every $i$. Suppose $G = \prod_{j=1}^{m}G_j^{c_i} \in \F_p[\bx]$ with $G_i\in \F_p[\bx]$ irreducible and $c_i\geq 1$. We say $F$ and $G$ {\em have no common component}, if $F_i \neq G_j$ for every pair of $i,j$. 
\begin{lem}(Corollary of B{\'e}zout's Theorem)\label{lem:bezout}
{\em
	Let $F, G\in \F_p[\bx]$ be two curves with no common components, then
$\sum_{\zeta} m_\zeta(F)m_\zeta(G)\leq \deg(F)\deg(G)$. } 
\end{lem}
Now let $F' = \prod_{i=1}^{l}F_i\in \F_p[\bx]$ be the square-free part of $F$. We say a singular point $\zeta$ on $F$ is {\it an isolated singular point} if $\zeta$ is also singular on $F'$, and call it {\it a non-isolated singular point} if otherwise. 
\begin{lem}\label{lem:d-1}
	{\em
	Let $F\in \F_p[\bx]$ be a curve with degree $d$, and let $F'$ denote the square-free part of $F$. Then 
	\begin{align*}
	\sum_{\zeta} m_\zeta(F')\left(m_\zeta(F')-1\right) \leq d(d-1)
	\end{align*}
	In particular, $F$ has at most ${d \choose 2}$ many isolated singular points. }
\end{lem}
\begin{proof}
	As $F'$ is squarefree, then $F'$ and $D^{1,0}F'(\bx)$ have no common component. It is also easy to deduct from Lemma \ref{lem:alt} that for any $\zeta\in \F_p^2$, $m_\zeta(D^{1,0}F') \geq m_\zeta(F') - 1$. The conclusion thus follows by applying Lemma \ref{lem:bezout}, and that $m_{\zeta}(F') \geq 2$ for any isolated singular points of $F$.
\end{proof}

Suppose $F = \prod_{i=1}^{l} F_i^{e_i}\in \F_p[\bx]$ is a nonconstant polynomial. For each $i$, let $d_i:=\deg(F_i)$ and let $d := \sum d_i^{e_i}$ be the total degree of $F$. Let $I\subseteq\{1,\ldots,l\}$ be an nonempty subset of indices, and let $S_I$ denote the set of points in the intersection $\bigcap_{i\in I}F_i$, and let $T_I = \{\zeta\in S_I: \zeta \text{ is smooth on } F_i \text{ for all }i\in I\}$. 

We then prove the following more generalized statement of Lemma \ref{lem:d-1} 
in the Appendix:
\begin{lem}\label{prp:summul}
{\em	Using the notation above we have: 
\begin{equation}
\label{eq:isosing}
\sum\limits_{\substack{\zeta\in S_I\\ I\neq \emptyset}}  
m_\zeta(F)(m_{\zeta}(F) - \sum_{i\in I}e_i) + 
\sum\limits_{\substack{\zeta\in T_I\\|I|\geq 2}} 
	\!\! m_{\zeta}(F) \leq d(d-1).
	\end{equation}} 
\end{lem}

Observe that if $\zeta\in S_I$ and $\zeta$ is an isolated singular points on 
$F$, then either $\zeta\in T_I$ or $m_{\zeta}(F) > 
\sum_{i\in I}\mu_\zeta(F_i)$, and $m_{\zeta}(F) = \sum_{i\in I}\mu(F_i)$ if 
it is non-isolated. So only the part corresponding to the isolated singular 
points contribute to the sum on the left hand side of Equation 
\ref{eq:isosing}. So we obtain the following: 
\begin{thm}{\em
	Let $f(\bx)\in \Z[\bx]$ be a nonconstant polynomial of degree $d$. 
Fix a prime $p$ and suppose that $\tf$ does not vanish identically over $\zp$. 
Then $\sum\limits_{\substack{\zeta\text{ isolated}\\\text{singular on }\tf}}  \!\!\!\! \deg\tf_{1,\zeta} \leq d(d-1)$.
}
\end{thm}
\begin{proof}
	This is immediate by observing that $\deg \tf_{1,\zeta} \leq s(f,\zeta)\leq m_{\zeta}(\tf)$. 
\end{proof}

However, bounding the degree of the perturbations $\tf_{1,\zeta}$ corresponding to {\em non-isolated singular points} of $\tf$ can be hard. This is evident in 
the discussion in the final section of the Appendix: lifting non-isolated 
singular points for certain families of curves requires extra care. 

\subsection{Algorithms and Complexity Analysis: Proof of Theorem \ref{thm:sepcurve}}\label{sec:algo}
For this section, let us consider bivariate polynomials $f(\bx) \in \Z[\bx]$ of 
the form $f(\bx) = g(x_1) + h(x_2)$. One broad family of examples of such 
bivariate polynomials is the family of superelliptic curves: 
$f(\bx) = x_2^d - g(x_1)$.

\begin{lem}\label{lem:sepcurve}
{\em	Let $F(x_1,x_2) = g(x_1) + h(x_2)\in \F_p[\bx]$ such that $g, h$ are nonconstant polynomials. Then $F$ is squarefree.}
\end{lem}
\begin{proof}
	Suppose $F$ is not squarefree and let $F = \prod_{i=1}^{l} F_i^{e_i}\in \F_p[\bx]$ be the irreducible factorization of $F$, and $e_i\geq 1$. Without loss of generality assume $e_1 > 1$ and $g'(x_1) = D^{1,0}F \neq 0$. Let $G = F/F_1^{e_1} = \prod_{i=2}^lF_i^{e_i}$. Differentiating $F$ with respect to $x_1$, we have
	\begin{align*}
	g'(x_1) &= e_1F_1^{e_1-1}D^{1,0}F\cdot G + F_1^{e_1} \cdot D^{1,0}G \\
	&= F_1^{e_1-1}\left(e_1D^{1,0}F\cdot G + F_1 \cdot D^{1,0}G\right). 
	\end{align*}
	So $F_1(x_1,x_2)$ must divide $g'(x_1)$, implying that $h(x_2)$ is a constant, a contradiction. 
\end{proof}

We now have enough ingredients to state our main algorithm: \\    
\mbox{}\hspace{.1cm}\scalebox{.84}[.84]{\fbox{\mbox{}\hspace{.1cm}\vbox{
\begin{algor}[{\tt PrimePowerPointCounting}$(f,p,k)$] 
\label{algor:curvemain}
\mbox{}\\ 
{\bf Input.} 
$(f,p,k)\!\in\!\Z[\bx]\times \N\times \N$ with 
$p$ prime and $f(\bx) = g(x_1)+h(x_2)$.\\ 
{\bf Output.} An integer $M\!\leq\!N_{p,k}(f)$ 
that, with probability at least $\frac{2}{3}$, is exactly $N_{p,k}(f)$.\\ 
{\bf Description.} \\ 
1: {\bf Let} $v\!:=s(f)$ and $f_{0,0}\!:=\!f$.\\  
2: {\bf If} $v\!\geq\!k$ \\ 
3: \mbox{}\hspace{.7cm}{\bf Let} $M\!:=\!p^{2k}$. {\bf Return}.\\   
4: {\bf If} $v\!\in\!\{1,\ldots,k-1\}$\\ 
5: \mbox{}\hspace{.7cm}{\bf Let} $M\!:=\!p^{2v}\text{{\tt 
PrimePowerPointCounting}$\left(\frac{f_{0,0}(\bx)}{p^v},p,k-v\right)$}$. 
	{\bf Return}.\\ 
6: {\bf End(If)}. \\ 
7: {\bf If} $s(g) = s(h) = 0$ \\
8: \mbox{}\hspace{.7cm} {\bf Let} $M\!:=p^{k-1}n_p(f)$. \\ 
9: \mbox{}\hspace{.7cm} {\bf For} $\zeta^{(0)}\!\in\!(\zp)^2$ a singular point 
of $\tf_{0,0}$ {\bf do}$^2$\\ 
\mbox{}\hspace{-.2cm}10: \mbox{}\hspace{1.4cm}{\bf Let} $s\!:=\!s(f_{0,0},\zeta^{(0)})$.\\   
\mbox{}\hspace{-.2cm}11: \mbox{}\hspace{1.4cm}{\bf If} $s\!\geq\!k$\\ 
\mbox{}\hspace{-.2cm}12: \mbox{}\hspace{2.1cm}{\bf Let} $M\!:=\!M+
 p^{2(k-1)}$.\\ 
\mbox{}\hspace{-.2cm}13: \mbox{}\hspace{1.4cm}{\bf Elseif} $s\!\in\{2,\ldots,
 k-1\}$\\ 
\mbox{}\hspace{-.2cm}14: \mbox{}\hspace{2.1cm}{\bf Let} 
  $M\!:=\!M+p^{2(s-1)}\text{{\tt PrimePowerPointCounting}$\left(f_{1,\zeta^{(0)}},p,k-s\right)$}$.\\ 
\mbox{}\hspace{-.2cm}15: \mbox{}\hspace{1.4cm}{\bf End(If)}.\\   
\mbox{}\hspace{-.2cm}16: \mbox{}\hspace{0.7cm}{\bf End(For)}.\\
\mbox{}\hspace{-.2cm}17: {\bf Elseif} $s(g) \geq 0$ or $s(h) \geq 0$ accordingly\\ 
\mbox{}\hspace{-.2cm}18: \mbox{}\hspace{.7cm} {\bf Let} $M\!:=p^{k}n_p(g)$ or $p^kn_p(h)$. \\  
\mbox{}\hspace{-.2cm}19: \mbox{}\hspace{.7cm} {\bf For} $\zeta^{(0)}\!\subseteq\!(\zp)^2$ a set of singular points of $\tf_{0,0}$ from a degenerate root of $\tilde{g}$ or $\tilde{h}$ {\bf do}\\  
\mbox{}\hspace{-.2cm}20: \mbox{}\hspace{1.4cm}{\bf Let} $s\!:=\!s(f_{0,0},\zeta^{(0)})$.\\   
\mbox{}\hspace{-.2cm}21: \mbox{}\hspace{1.4cm}{\bf If} $s\!\geq\!k$\\ 
\mbox{}\hspace{-.2cm}22: \mbox{}\hspace{2.1cm}{\bf Let} $M\!:=\!M+p^{2k-1}$.\\ 
\mbox{}\hspace{-.2cm}22: \mbox{}\hspace{1.4cm}{\bf Elseif} $s\!\in\{2,\ldots,
k-1\}$\\ 
\mbox{}\hspace{-.2cm}23: \mbox{}\hspace{2.1cm}{\bf Let} 
$M\!:=\!M+p^{2s-1}\text{{\tt 
PrimePowerPointCounting}$\left(f_{1,\zeta^{(0)}},p,k-s\right)$}$.\\ 
\mbox{}\hspace{-.2cm}24: \mbox{}\hspace{1.4cm}{\bf End(If)}.\\   
\mbox{}\hspace{-.2cm}25: \mbox{}\hspace{0.7cm}{\bf End(For)}.\\
\mbox{}\hspace{-.2cm}26: {\bf End(If)}.\\
\mbox{}\hspace{-.2cm}27: {\bf Return}.
\end{algor}}
}
}

There are some remaining details to clarify about our algorithm. First,  
let $s(f)$ denote the largest power of $p$ that divides all the coefficients 
of $f$. Then by Definition \ref{dfn:ntpk}, we see that any polynomial in 
$T_{p,k}(f)$ should also be of the form $g(x_1) + h(x_2)$ with $s(g) = 0$ or 
$s(h) = 0$. By Lemma \ref{lem:sepcurve}, we see that when $s(g) = s(h) = 0$, 
then $\tf \mod p$ is squarefree. Now without loss of generality, suppose 
$0 = s(g) < s(h) = c$, then $\tf(\bx) = \tilde{g}(x_1)$ mod $p$. Then any 
singular point on $\tf$ should be of the form $(\zeta^{(0)}_1, y)$ for any 
{\it degenerate root} $\zeta^{(0)}_1$ of the univariate polynomial 
$\tilde{g}(x_1)\in \F_p[x_1]$ and any choice of $y\in \{0,1,\ldots,p-1\}$. So 
it makes sense to consider the perturbation of $f$ in the direction of $x_1$ 
only.

Let $\zeta^{(0)}_1$ be any degenerate root of $\tilde{g}$. Abusing 
notation, let $\zeta^{(0)}:=\{\zeta^{(0)}_1\}\times \F_p = \{(\zeta^{(0)}_1, y): y\in \{0,1,\ldots,p-1\}\}$, the set of singular points of $\tf$ with the first coordinate being $\zeta^{(0)}_1$. Consider
$f(\zeta^{(0)}_1 + px_1,x_2) = g(\zeta^{(0)}_1 + px_1) + h(x_2)$. 
Let $s(f, \zeta^{(0)}):=s(f(\zeta^{(0)}_1 + px_1,x_2)) = \min(s(g,\zeta^{(0)}_1), c)$, the largest $p$'s power dividing all the coefficients of the 
perturbation, and let 
$f_{1,\zeta^{(0)}} = \frac{1}{p^{s(f, \zeta^{(0)})}}f(\zeta^{(0)}_1 + 
px_1,x_2)$.

We prove the following more specific version of Lemma \ref{lem:nrec} in 
the Appendix:
\begin{lem}\label{lem:sepbasic}
{\em
	Let $f(\bx) = g(x_1) + h(x_2)$ with $0 = s(g) < s(h) = c$. Let $n_p(g)$ denote the number of {\it non-degenerate root} of $\tilde{g}$ in $\F_p$, and following the notation above: 
	\begin{align*}
	N_{p,k} = p^kn_p(g)+\left(\sum\limits_{\substack{\zeta^{(0)} \subseteq (\Z/p\Z)^2\\
			s(f, \zeta^{(0)})\geq k}} p^{2k-1}\right) 
	+ 
	\sum\limits_{\substack{\zeta^{(0)}\subseteq(\Z/p\Z)^2\\
			s(f,\zeta^{(0)})\leq k-1}} p^{2s(f,\zeta^{(0)})-1}N_{p,k-s(f,\zeta^{(0)})}(f_{1,\zeta^{(0)}})
	\end{align*}}
\end{lem}

By symmetry, a variant of our preceding lemma also holds when 
$0 = s(h) < s(g) = c$. 
Similarly, for any degenerate root $\zeta^{(0)}_2$ of the univariate polynomial $\tilde{h}(x_2)\in \F_p$, we denote $\zeta^{(0)}:=\F_p\times\{\zeta^{(0)}_2\}$ to be the set of singular points of $\tf$ with the second coordinate being $\zeta^{(0)}_2$. 

\begin{nota}{\em
	Suppose $\zeta^{(i-1)} = \{\zeta^{(i-1)}_1\}\times \F_p$ is the set of singular points on $\tf_{i-1,\zeta}$ for some polynomial in $T_{p,k}(f)$ and $\zeta = (\zeta_1,\zeta_2)$, we write
	\begin{align*}
	\zeta + p^{i-1}\zeta^{(i-1)} = \{(x_1, x_2): x_1 = \zeta_1 + p^{i-1}\zeta^{(i-1)}_1, x_2\in\{\zeta_2 + p^{i-1}\cdot 0, \ldots \zeta_2 + p^{i-1}\cdot p-1\}\}
	\end{align*}
	as element-wise operations for set. We also use this notation similarly when $\zeta^{(i-1)} = \F_p\times\{\zeta^{(i-1)}_2\}$.}
\end{nota}

We are now ready to prove the correctness of our main algorithm.

\medskip
\noindent 
{\bf Proof of Correctness of Algorithm \ref{algor:curvemain}:} 
Assume temporarily that Algorithm \ref{algor:curvemain} is correct 
when $s(f)=0$, i.e. when $f_{0,0}$ is {\em not} identically $0$ mod $p$. 
Since for any integers $a$ with $a\!\leq\!k$, and any elements $\bx,\mathbf{y} \in (\zpk)^2$, 
$p^a \bx\!=\!p^a \mathbf{y}$ mod $p^k \Longleftrightarrow 
\bx\!=\!\mathbf{y}$ mod $p^{k-a}$, Steps 1--6 of our algorithm then  dispose of the case where $f$ is identically $0$ in 
$(\zp)[\bx]$. So let us now prove correctness when $f$ is {\em not} 
identically $0$ in $(\zp)[\bx]$.

Recall from the discussion at the very beginning of this section, we see that any polynomial in $\tpk(f)$ should be of the form $f_{i,\zeta^{(i-1)}}(\bx):=g_i(x_1)+h_i(x_2)$ with $s(g_i) = 0$ or $s(h_i) = 0$. Applying Lemma \ref{lem:nrec} and Lemma \ref{lem:sepbasic} accordingly, we then 
see that 
it is enough to prove that the value of $M$ is the value of our formula for 
$N_{p,k}(f)$ when the two {\bf For} loops of Algorithm \ref{algor:curvemain} runs 
correctly. 

When $s(g) = s(h) = 0$, Steps 7--16 (once the {\bf For} loop is completed) then simply 
add the second and third summands of our formula in Lemma \ref{lem:nrec} to 
$M$ thus ensuring that $M\!=\!N_{p,k}(f)$. On the other hand, when $s(g)> 0$ or $s(h)>0$, Steps 17--26 (once the {\bf For} loop is completed) handles
add the second and third summands of our formula in Lemma \ref{lem:sepbasic} to 
$M$ thus ensuring that $M\!=\!N_{p,k}(f)$. So we are done. \qed 

\vspace{10pt}

In \cite{krrz19}, we defined a recursive tree structure for root counting for 
univariate polynomial in $\zpk$. We define similarly a recursive tree for $f(\bx) = g(x_1)+h(x_2)$ that will enable our complexity analysis. 
\begin{dfn}\label{dfn:septree}{\em
	Let us identify the elements of $T_{p,k}(f)$ with nodes of a lablled rooted directed tree $\ctpk(f)$. 
	\begin{enumerate}
		\item[(1)]{We set $f_{0,0}\!:=\!f$, $k_{0,0}\!:=\!k$, and let 
			$(f_{0,0},k_{0,0})$ be the label of the root node of 
			$\cT_{p,k}(f)$. }   
		\item[(2)]{There is an edge from node $(f_{i',\zeta'},k_{i',\zeta'})$ to 
			node $(f_{i,\zeta},k_{i,\zeta})$ if and only if 
			$i'\!=\!i-1$ and there is a (set of) singular points $\zeta^{(i-1)}$ in $(\zp)^2$ of $\tf_{i',\zeta'}$ with $s(f_{i',\zeta'},\zeta^{(i-1)})\leq k_{i',\zeta'}-1$ and $\zeta\!=\!\zeta'+p^{i-1}\zeta^{(i-1)}$ in
			$(\Z/p^i\Z)^2$.}  
		\item[(3)]{Suppose $f_{i',\zeta'} = g_{i'}(x_1) + h_{i'}(x_2)$. The label of a directed edge from node $(f_{i',\zeta'},k_{i',\zeta'})$ 
			to node $(f_{i,\zeta},k_{i,\zeta})$ is $p^{2\left( s\left(f_{i',\zeta'},
				(\zeta-\zeta')/p^{i'}\right)-1\right) }$ or $p^{2s\left(f_{i',\zeta'},
				(\zeta-\zeta')/p^{i'}\right)-1 }$ respectively when $s(g_{i'}) = s(h_{i'}) = 0$ or otherwise.} 
	\end{enumerate}  
In particular, the labels of the nodes lie in $\Z[\bx]\times \N$. }
\end{dfn}

\begin{rem} \mbox{}\\ 
{\em
1. Just as the tree structure for the univariate polynomial in \cite{krrz19}, 
our trees $\ctpk(\cdot)$ encode algebraic expressions for our desired root 
counts $N_{p,k}(\cdot)$. In particular, the children of a node labelled 
$(f_i,k_i)$ yield terms that one sums to get the root count $N_{p,k_i}(f_i)$, 
and the edge labels yield weights multiplying the corresponding terms. 

\smallskip 
\noindent 
2. One main difference is that the correspondence between polynomials in 
$\tpk(f)$ with the label in the tree $\ctpk(f)$ is no longer one-to-one. In 
particular, in the case when $f_{i,\zeta}(\bx)=g_i(x_1)+h_i(x_2)$ with 
$s(g_i) >0$, its child node polynomial $f_{i+1, \zeta'}$ for $\zeta'-\zeta = 
\{\zeta^{(i)}_1\}\times \F_p$, correspond to {\em a set of } singular points of 
$\tf_{i,\zeta}$ with the first coordinate equaling to a degenerate root 
$\zeta^{(i)}_1$ of $\tilde{g}_i$. \dia }
\end{rem}
The following lemma, proved in the Appendix, will be central in our complexity 
analysis. 
\begin{lem}\label{lem:curvebranch}{\em
	Let $f(\bx) = g(x_1) + h(x_2)\in \Z[\bx]$ be a nonconstant polynomial of degree $d$. Following the notation of Definition \ref{dfn:septree}, we have 
that:
	\begin{enumerate} 
		\item[(1)]{The depth of $\cT_{p,k}(f)$ is at most $k$.} 
		\item[(2)]{The degree of the root node of $\cT_{p,k}(f)$ is at most ${d\choose 2}$.} 
		\item[(3)]{The degree of any {\em non}-root node of $\cT_{p,k}(f)$ labeled 
			$(f_{i,\zeta},k_{i,\zeta})$, with parent $(f_{i-1,\mu},k_{i-1,\mu})$ 
			and \ $\zeta^{(i-1)}:=(\zeta-\mu)/p^{i-1}$, \ 
			is \ at \ most \ $s(f_{i-1,\mu},\zeta^{(i-1)})$. \ In \ 
			particular,\\  
			$\deg \tf_{i,\zeta}\!\leq\!s(f_{i-1,\mu},\zeta^{(i-1)})
			\!\leq\!k_{i-1,\mu}-1\!\leq\!k-1$ and
			\begin{align*}
			\sum\limits_{\substack{(f_{i,\zeta},k_{i,\zeta}) \text{ a child}\\ 
					\text{of } (f_{i-1,\mu},k_{i-1,\mu})}}  \!\!\!\!\!
			\!\!\!\!\!
			\deg \tf_{i,\zeta}\left( \deg \tf_{i,\zeta}-1\right) \leq \deg\tf_{i-1,\mu}\left( \deg\tf_{i-1,\mu} - 1\right) 
			\end{align*}
		}
		\item[(4)]{$\cT_{p,k}(f)$ has at most ${d \choose 2}$ nodes at 
			depth $i\!\geq\!1$, and thus a total of no more than 
			$1+(k-1){d\choose 2}$ nodes.}  
	\end{enumerate} }
\end{lem}

\noindent 
{\bf Proof of Theorem \ref{thm:sepcurve}:} 
Since we already proved that Algorithm \ref{algor:curvemain}, it suffices to prove the stated complexity bound for Algorithm \ref{algor:curvemain}. The proof consists of three parts: (a) the point counting algorithm over $\F_p$ from 
\cite{har15}, (b) the univariate reduction and the factorization algorithm, 
and (c) applying Lemma \ref{lem:curvebranch} to show that the number of 
necessary factorization and point counting, and $p$-adic valuation calculations is well-bounded.

More specifically the {\bf For} loops and the recursive calls of Algorithm \ref{algor:curvemain} can be seen as the process of building the tree $\ctpk(f)$. We begin at the root node by applying the algorithm in \cite{har15} to find the 
number of roots of $\tf$ in $\F_p$. This computation takes time $O(d^8p^{1/2}\log^{2+\varepsilon}p)$ and space $O(d^4p^{1/2}\log p)$ by \cite{har15}. 
(Specifically, one avails to Theorem 3.1, Lemmata 3.2 and 3.4, and Proposition 
4.4 from Harvey's paper.) 

To find singular points of $\tf$, it suffices to find the roots of the 
$2\times 2$ polynomial system $F := (\tf(\bx), D^{1,0}\tf(\bx))$ 
over $\F_p$. This is done by first transforming the problem to factorization 
of a univariate polynomial $U_F$ via univariate reduction over the finite 
field (see, e.g. \cite{roj99}). In particular $\deg U_F\leq d^2$ and roots of $U_F$ will encode information on tuple $(x_1,x_2)$ as solutions to the polynomial system $F$. Computing $U_F$ can be done in time polynomial in the 
{\em mixed area} of the Newton polygons of $F$, and takes time $\tilde{O}(d^{15})$ and space $O(d^4)$ (\cite{roj99}). Then we use the fast randomized Kedlaya-Umans factoring algorithm 
in \cite{ku08} to find solutions to $U_F$ in $\F_p$, and thereby the singular 
points of $\tf$. This takes time $(d^3\log p)^{1+o(1)} + (d^2\log^2p)^{1+o(1)}$ 
and requires $O(d^2\log p)$ random bits.

In order to continue the recursion, we need to compute $p$-adic valuations of polynomial coefficients to determine $s(f_{0,0},\zeta^{(0)})$ and the edges emanating from the root node. Expanding $f(\zeta^{(0)}+p\bx)\mod p^k$ takes time no worse than $d^2(k\log p)^{1+o(1)}$ via Horner's method and fast finite ring 
arithmetic (see, e.g., \cite{bs,vzgbook}). Computing $s(f_{0,0},\zeta^{(0)})$ thus takes time $d(k\log p)^{1+o(1)}$ by evaluating $p$-adic valuations using standard tools such as binary methods. By Assertion (2) of Lemma \ref{lem:curvebranch}, there are no more than ${d\choose 2}$ many such $\zeta^{(0)}$. So the total work so far is $d^{15+\varepsilon}(k\log p)^{1+o(1)}p^{1/2+\varepsilon}$. Note that computing the univariate reduction $U_F$ and 
$N_{p,1}(f)$ 
via algorithm in \cite{har15} dominates the computation.

The remaining work can also be well-bounded similarly by Lemma 
\ref{lem:curvebranch}. In particular, the sum of the degress if 
$\tf_{i,\zeta}$ at level $i$ of the tree $\ctpk(f)$ is no greater than 
${d\choose 2}$. 

Now observe that for $i\geq 2$, the amount of work needed to determine the polynomials at level $i$ via computing $s(f_{i-1,\mu},\zeta^{(i-1)})$ is no greater than ${d \choose 2} d(k\log p)^{1+o(1)}$. As $\deg \tf_{1,\zeta}\leq d$ for every $f_{i,\zeta}$ in the tree $\ctpk(f)$ and there are at most ${d\choose 2}$ many such polynomials for each $i\geq 1$, the total amount of work for point counting over $\F_p$, univariate reduction and factorization for each subsequent level of $\ctpk(f)$ will be $d^{17+\varepsilon}(k\log p)^{1+o(1)}p^{1/2+\varepsilon}$ with $O(d^2\log p)$ random bits needed. The expansion of the $f_{i,\zeta}$ at level $i$ will take time no greater than $d^3(k\log p)^{1+o(1)}$ to compute. So the total work at each subsequent level is $d^{17+\varepsilon}(k\log p)^{1+o(1)}p^{1/2+\varepsilon}$.

Therefore the total amount of work for our tree will be $d^{17+\varepsilon}(k\log p)^{2+\varepsilon}p^{1/2+\varepsilon}$, and the number of random bits needed is $O(d^2k\log p)$.

The argument proving the Las Vegas properties of our algorithm can be done similarly as in \cite{krrz19}. In particular, we run factorization algorithm for sufficiently many times to reduce the {\em overall} error probability to less than $2/3$. Thanks to Lemma \ref{lem:curvebranch}, it is enough to enforce a success probability of $O(\frac{1}{d^2k})$ for each application of factorization, and to run the algorithm from \cite{ku08} for $O(\log(dk))$ times for each time we need univariate factorization. So a total of $O(d^2k\log(dk)\log p)$ many random bits is needed.  

Our algorithm proceeds with building the tree structure $\ctpk(f)$, so we only need to keep track of collections of $f_{i,\zeta}$. A bivariate polynomial of degree $d$ with integer coefficients all of absolute value less than $p^k$ requires $O(dk\log p)$ bits to store, and there are no more than ${d\choose 2}k$ many polynomials in $\ctpk(f)$. Combining with the space needed from algorithm in \cite{har15}, we only need $O(d^4kp^{1/2}\log p)$ space. 

If $\tf$ defines a smooth and irreducible curve over the algebraic closure 
$\bar{\F}_p$ of $\F_p$ then the second part of the theorem follows immediately 
by combining our bivariate version of Hensel's Lemma (Lemma \ref{lem:nhensel}) 
with Kedlaya's quantum point counting algorithm from \cite{kedlaya}. \qed 

\section{Appendix: Remaining Proofs and Finessing Exceptional Curves} 

\subsection{Proof of Lemma \ref{lem:nhensel} (Higher-Dimensional Hensel's Lemma)} 
Consider the Taylor expansion of $f$ at $\sigma$ by $p^j\bx$,
\begin{align*}
        f(\sigma+p^j\bx) &= f(\sigma) + p^j\left( \frac{\partial f}{\partial x_1}(\sigma)x_1  + \frac{\partial f}{\partial x_2}(\sigma)x_2\right) + \sum_{i_1+i_2\geq 2}p^{j(i_1+i_2)} D^{i_1,i_2}f(\sigma)x_1^{i_1} x_2^{i_2}\\
        &=f(\sigma) + p^j\left( \frac{\partial f}{\partial x_1}(\sigma)x_1 + \frac{\partial f}{\partial x_2}(\sigma)x_2\right) \mod p^{j+1},
        \end{align*}
        as $j(i_1+i_2) \geq j+1$ for all $i_1+i_2\geq 2$. Then $\bt:=(t_1,t_2)$ is such that $(\sigma + \bt p^j)$ is a solution to $f\equiv 0\mod p^{j+1}$ if and only if
        \begin{align} \label{eq:hensel}
        \frac{\partial f}{\partial x_1}(\sigma)t_1 + \frac{\partial f}{\partial x_2}(\sigma)t_n = -\frac{f(\sigma)}{p^j} \mod p.
        \end{align}
        As $(\zeta^{(0)} = \sigma \mod p)$ is a smooth point on $\tf$, then 
there exists an $i$ such that $\frac{\partial f}{\partial x_i}(\sigma) = 
\frac{\partial f}{\partial x_i}(\zeta^{(0)})\neq 0 \mod p$. Then left hand side 
of (\ref{eq:hensel}) does not vanish identically, and thus define a nontrivial 
linear relation in $(\Z/p\Z)^2$. So fixing $\zeta$, there are exactly $p$ many 
$\bt\in (\Z/p\Z)^2$ satisfying (\ref{eq:hensel}). \qed

\subsection{The Proof of Lemma \ref{prp:summul}}   
We prove by induction on the number of irreducible components of $F$. 
	
	When $l = 1$, $F = F_1^{e_1}$. By Lemma \ref{lem:func}, $m_{\zeta}(F) = e_1m_{\zeta}(F_1)$ for every $\zeta\in \F_p^2$. Then by Lemma \ref{lem:d-1} and expanding
	\begin{align*}
	\sum_{\zeta \text{ on }F_1} \frac{m_\zeta(F)}{e_1}\left( \frac{m_\zeta(F)}{e_1} - 1\right) \leq d_1(d_1-1), 
	\end{align*}
	the conclusion holds.

	Now suppose the inequality holds for $l-1>1$, and let $F' = \prod_{i=1}^{l-1} F_i^{e_i}$ and $d'$ be its degree, and $F_l$ is irreducible and has no common component with $F'$. Then $\sum_{\zeta \text{ on }F_l} m_\zeta(F_l^{e_l})\left( m_\zeta(F_l^{e_l}) - e_l\right) \leq e_ld_l(e_ld_1-e_l)$, and
	\begin{align*}
	\sum_{J\subseteq \{1,\ldots,l-1\}}\left[  \sum_{\zeta\in S_J}m_\zeta(F')\left(m_{\zeta}(F') - \sum_{j\in J}e_j\right) +\sum\limits_{\substack{\zeta\in T_J\\|J|\geq 2}}  
	\!\! m_{\zeta}(F') \right] \leq d'(d'-1)
	\end{align*} 
	By Lemma \ref{lem:bezout}, we must have $\sum_{\zeta} m_\zeta(F')m_\zeta(F_l^{e_l})\leq d'd_le_l$. Summing over all $J\subseteq \{1,\ldots,l-1\}$, we have
	{\small
		\begin{align*}
		&\sum_{J}\left[  \sum_{\zeta\in S_J}m_\zeta(F')\left(m_{\zeta}(F') - \sum_{j\in J}e_j\right) +\sum\limits_{\substack{\zeta\in T_J\\|J|\geq 2}} \!\! m_{\zeta}(F') \right] + 2\sum_{\zeta} m_\zeta(F')m_\zeta(F_l^{e_l})+ \sum_{\zeta \text{ on }F_l} m_\zeta(F_l^{e_l})\left( m_\zeta(F_l^{e_l}) - e_l\right) \\
		&\leq  d'(d'-1) + 2d'd_le_l + (d_le_l)^2-e_l^2d_l \leq  (d'+d_le_l)^2  -d'-e_l^2d_l\leq d(d-1).
		\end{align*}
	}
	Note that for each $J\subseteq\{1,\ldots,l-1\}$ and each $\zeta\in S_J$ such that $\zeta$ is not a point of $F_l$, $m_\zeta(F') = m_\zeta(F)$. If $\zeta\in S_{J\cup \{l\}}\backslash T_{J\cup \{l\}}$, then $m_{\zeta}(F_l^{e_l}) + m_{\zeta}(F') > e_l + \sum_{j\in J}e_j$, and
	\begin{align*}
	&m_\zeta(F')(m_{\zeta}(F') - \sum_{i\in J}e_i) + 2m_\zeta(F_l^{e_l})m_p(F') +  m_\zeta(F_l^{e_l})( m_\zeta(F_l^{e_l}) - e_l) \\
	&=(m_\zeta(F') +m_\zeta(F_l^{e_l}))^2 -\sum_{i\in J}e_im_\zeta(F') - e_lm_\zeta(F_l^{e_l})\\
	&\geq m_\zeta(F)( m_\zeta(F) - \sum_{i\in J\cup\{l\}}e_i)                                                                                                                          
	\end{align*}
	So we can rewrite
	{\small
		\begin{align*}
		&A:=\sum_{J\subseteq \{1,\ldots,l-1\}} \left[ \sum_{\zeta\in S_J}m_\zeta(F')(m_{\zeta}(F') - \sum_{j\in J}e_j) + 2\sum_{\zeta\not\in T_{J\cup\{l\}}} m_\zeta(F')m_\zeta(F_l^{e_l})\right] + \sum_{\zeta\in S_{\{l\}}} m_\zeta(F_l^{e_l})\left( m_\zeta(F_l^{e_l}) - e_l\right) \\
		&\geq \sum_{J\subseteq \{1,\ldots,l-1\}} \left[ \sum_{\zeta\in S_J} m_\zeta(F)(m_{\zeta}(F) - \sum_{j\in J}e_j) +\sum_{\zeta \in S_{J\cup\{l\}}} m_\zeta(F)( m_\zeta(F) - \sum_{j\in J\cup\{l\}}e_i)\right]  +\sum_{\zeta \in S_{\{l\}}}m_\zeta(F)(m_{\zeta}(F) - e_l) \\
		&=\sum_{I\in \{1,\ldots,l\}}\sum_{\zeta\in I} m_\zeta(F)(m_{\zeta}(F) - \sum_{i\in I}e_i).
		\end{align*}
	}
	On the other hand, if $\zeta\in T_{J\cup\{l\}}$, we must have $m_{\zeta}(F_l^{e_l}) + m_{\zeta}(F') = e_l + \sum_{j\in J}e_j$. Then summing over all $J\subseteq \{1,\ldots,l-1\}$, and 
	\begin{align*}
	B:=&\sum_{J}\sum\limits_{\substack{\zeta\in T_J\\|J|\geq 2}}  
	\!\! m_{\zeta}(F') + 2\sum_{\zeta\in T_{J\cup\{l\}}}m_{\zeta}(F')m_{\zeta}(F_l^{e_l})\\
	&=\sum_{J}\left[ \sum\limits_{\substack{\zeta\in T_J\\|J|\geq 2}}  
	\!\! m_{\zeta}(F) + 2\sum\limits_{\substack{\zeta\in T_{J\cup\{l\}}\\|J|\geq 2}}  
	\!\! m_{\zeta}(F')m_{\zeta}(F_l^{e_l})\right] + \sum_{i=1}^{l-1}\sum_{\zeta\in T_{\{i,l\}}}m_{\zeta}(F')m_{\zeta}(F_l^{e_l}) \\
	&\geq \sum_{J}\left[ \sum\limits_{\substack{\zeta\in T_J\\|J|\geq 2}}  
	\!\! m_{\zeta}(F) + \sum\limits_{\substack{\zeta\in T_{J\cup\{l\}}\\|J|\geq 2}} 
	\!\! m_{\zeta}(F)\right] + \sum_{i=1}^{l-1}\sum_{\zeta\in T_{\{i,l\}}}m_{\zeta}(F') = \sum_I\sum\limits_{\substack{\zeta\in T_I\\|I|\geq 2}}  
	\!\! m_{\zeta}(F).
	\end{align*}
	The last inequality holds because for $a, b \geq 1$, we must have $2ab \leq a+b$.
	
	Combining all of above computations, we have
	\begin{align*}
	\sum_{I}\left[ \sum_{\zeta\in S_I} m_\zeta(F)(m_{\zeta}(F) - \sum_{i\in I}e_i) + \sum\limits_{\substack{\zeta\in T_I\\|I|\geq 2}}  
	\!\! m_{\zeta}(F) \right]  \leq A+B\leq d(d-1).
	\end{align*}

\noindent 	
The conclusion thus follows. \qed 

\subsection{The Proof of Lemma \ref{lem:curvebranch}} \mbox{}\\  
{\bf Assertion (1):} By Definitions \ref{dfn:ntpk} and \ref{dfn:septree}, each $(f_{i,\zeta},k_{i,\zeta})$ whose parent node is $(f_{i-1,\mu},k_{i-1,\mu})$, must satisfies $1\!\leq\!k_{i-1,\mu}-k_{i,\zeta}\!\leq\!k_{i-1,\mu}-1$, and $1\!\leq\!k_{i,\zeta}\!\leq\!k-1$ for all $i\!\geq\!1$. 
	So considering any root to leaf path in $\cT_{p,k}(f)$,   
	it is clear that the depth of $\cT_{p,k}(f)$ can be no greater than 
	$1+(k-1) = k$.
	
	\medskip 
	\noindent 
	{\bf Assertion (2):} If $s(g) = s(h) = 0$, then by Lemma \ref{lem:sepcurve}, $\tf(\bx)\in \F_p[\bx]$ is square-free. As the multiplicity of any singular point is at least $2$, by Lemma \ref{lem:d-1}, $\tf$ has at most ${d\choose 2}$ many singular points. In this case, each edge emanating from the root of $\ctpk(f)$ corresponds to a unique singular point of $\tf_{0,0}$.
	
	Suppose otherwise, and without loss of generality $0 = s(g) < s(h) = c$, then each edge emanating from the root node correspond to the set $\{\zeta^{(0)}_1\}\times \F_p$ for a unique degenerate root $\zeta^{(0)}_1$ of the univariate polynomial $\tilde{g}(x_1)$. As $\tilde{g}$ has at most $\floor{\frac{\deg\tilde{g}}{2}} \leq \floor{\frac{d}{2}} \leq {d\choose 2}$ degenerate roots, we are done.
	
\medskip 
\noindent 
{\bf Assertion (3):}  
Suppose $f_{i-1,\mu} = g_{i-1}(x_1) + h_{i-1}(x_2)\in \Z[\bx]$ with $s(g_{i-1}) = s(h_{i-1})=0$. Then $\zeta^{(i-1)}$ is a singular point of $\tf_{i-1,\mu}$, and let	
\begin{align*}
s:=s(f_{i-1,\mu},\zeta^{(i-1)}) = \min_{0\leq i_1+i_2 \leq k_{i,\zeta-1}}\left\lbrace (i_1+i_2) + \ord_p\left(D^{i_1,i_2}f_{i-1,\mu}(\zeta^{(i-1)})\right) \right\rbrace 
\end{align*} 
So then for each pair of $(\ell_1,\ell_2)$ with $\ell_1+\ell_2 \geq s+1$, the coefficient of $x^{\ell_1}_1x^{\ell_2}_2$ in the perturbation $f_{i-1,\mu}(\zeta^{(i-1)}+p\bx)$ must be divisible by $p^{s+1}$. In other words, the coefficient of $x^{\ell_1}_1x^{\ell_2}_2$ in $f_{i,\zeta}(\bx)$ must be divisible by $p$. So $\deg\tf_{i,\zeta}\leq s$.

	Now by Lemma \ref{lem:nmul}, we know that the multiplicity of $\zeta^{(i-1)}$ on $\tf_{i-1,\mu}$: $m_{\zeta^{(i-1)}}(\tf_{i-1,\mu}) \geq s(f_{i-1,\mu},\zeta^{(i-1)})$. Combining with \ref{lem:d-1}, we have
	\begin{align*}
	\sum\limits_{\substack{(f_{i,\zeta},k_{i,\zeta}) \text{ a child}\\ 
			\text{of } (f_{i-1,\mu},k_{i-1,\mu})}}  \!\!\!\!\!
	\!\!\!\!\!
	\deg \tf_{i,\zeta}\left( \deg \tf_{i,\zeta}-1\right) &\leq \sum\limits_{\substack{\zeta^{(i-1)} \text{ sing.}\\ 
			\text{point on } \tf_{i-1,\mu}}}m_{\zeta^{(i-1)}}(\tf_{i-1,\mu}) \left( m_{\zeta^{(i-1)}}(\tf_{i-1,\mu})  - 1\right) \\
	&\leq  \deg\tf_{i-1,\mu}\left( \deg\tf_{i-1,\mu} - 1\right). 
	\end{align*}

Suppose without loss of generality, $0 = s(g_{i-1}) < s(h_{i-1}) = c$. Then by a similar argument $\deg\tf_{i,\zeta}\leq s(f_{i-1,\mu},\zeta^{(i-1)}) = \min(s(\tilde{g},\zeta^{(i-1)}_1), c) \leq s(\tilde{g},\zeta^{(i-1)}_1)$. By 
Lemma \ref{lem:curvebranch} we have that 
$\sum\limits_{\substack{\zeta^{(i-1)}_1 \text{ a deg.}\\ 
			\text{root of } \tilde{g}_{i-1}}}  
	\!\!\!\!\!
	s(\tilde{g}_{i-1},\zeta^{(i-1)}_1)\leq \deg \tilde{g}_{i-1}$, so 
then $\sum\limits_{\substack{(f_{i,\zeta},k_{i,\zeta}) \text{ a child}\\ 
			\text{of } (f_{i-1,\mu},k_{i-1,\mu})}}\deg\tf_{i,\zeta}\leq \deg \tilde{g}_{i-1}$. We are done, simply by observing that for 
$\deg\tf_{i,\zeta} \geq 2$ and any collections of $a_i > 2$, we must have 
$\sum a_i(a_i-1)\leq \left( \sum a_i\right)\left( \sum a_i-1\right)$.

\medskip 
\noindent 
{\bf Assertion (4):} This is immediate from Assertions (1) and (3). \qed 

\subsection{The Proof of Lemma \ref{lem:sepbasic}}   
Any points over $\F_p$ on $\tf(\bx)$ is nonsingular if and only if $D^{1,0}(\tf) = \tilde{g}'(x_1) \neq 0\mod p$, as $h(x_2)$ is identically $0$ mod $p$. In other words, any nonsingular point on $\tf$ should be of the form $(\zeta^{(0)}_1,y)$ where $\zeta^{(0)}_1$ is a non-degenerate root of $\tilde{g}$, and any choice of $y\in \{0,1,\ldots,p-1\}$. So the number of non-singular point on $\tf$ is: $n_p(f) = p\cdot n_p(g)$. Then the first summand in the equation is obvious by plugging into the first summand in Lemma \ref{lem:nrec}.

	Now suppose $\zeta_0 :=\zeta^{(0)}_1$ is a degenerate root of the univariate polynomial $\tilde{g}$, and $\zeta^{(0)} = \{\zeta_0\}\times \F_p$. Write $\sigma = \zeta_0 + p\tau$, where $\tau:=\zeta_1+\ldots+p^{k-2}\zeta_{k-1}\in \Z/p^{k-1}\Z$ via base-$p$ expansion. Then by definition 
$f(\zeta_0+px_1, x_2) = p^{s(f, \zeta^{(0)})}f_{1,\zeta^{(0)}}(x_1,x_2)$, 
where $f_{1,\zeta^{(0)}}\in\Z[x_1,x_2]$ does not vanish identically mod $p$. 
	
	If $k \geq s(f,\zeta^{(0)})$, then $f(\sigma,y) = 0\mod p^k$ regardless of choice of $\tau\in \Z/p^{k-1}\Z$ and $y\in \Z/p^k\Z$. So there are exactly $p^{k-1}\cdot p^k = p^{2k-1}$ many pairs of $(\sigma,y)\in (\zpk)^2$ such that $\sigma = \zeta_0 \mod p$ and $f(\sigma,y) = 0\mod p^k$. 
	
	If $s(f,\zeta^{(0)}) \leq k-1$, then $f(\sigma,y) = 0 \mod p^k$ if and only if 
	\begin{equation} \label{eq:xpert}
	f_{1,\zeta^{(0)}}(\tau, y) = 0\mod p^{k-s(f,\zeta^{(0)})}.
	\end{equation}
	Let $s:=s(f,\zeta^{(0)})$, then $\tau = \zeta_1+p\zeta_2+\ldots+p^{k-s-1}\zeta_{k-s}\mod p^{k-s}$ and $y:=\sum_{i=0}^{k-1}p^iy_i = y_0+\ldots+p^{k-s-1}y_{k-s-1}\mod p^{k-s}$. So the rest of the base-$p$ digits, $\zeta_{k-s+1},\ldots,\zeta_{k-1}$ and $y_{k-s},\ldots,y_{k-1}$ respectively does not appear in 
Equality (\ref{eq:xpert}). The possible lifts $\zeta$ where the first 
coordinate mod $p$ is $\zeta_0$ is thus exactly $p^{s-1}\cdot p^{s}$ times the 
number of roots $(\tau,y)\in (\Z/p^{k-s}\Z)^2$ of $f_{1,\zeta^{(0)}}$. \qed 

\subsection{Exceptional Curves}
Let $f(\bx)\in \Z[\bx]$ be a nonconstant polynomial, and let $s(f)$ denote the largest $p$-th power dividing all the coefficients of $f$.

Consider $f(\bx) = g^{d}(\bx) +p^{cd}h(\bx)\in \Z[\bx]$, with $d\geq 2$ and $c\geq 1$. Moreover, $f(\bx) \equiv g^{d}(\bx) \mod p$ and $f$ is irreducible mod $p$.

For $k\leq cd$, $f(\bx) = g^{d}(\bx) \mod p^{k}$. Now suppose $\zeta^{(0)}$ is 
a smooth point on $(g\mod p)$. Then by Hensel's Lemma (Lemma 
\ref{lem:nhensel}), $\zeta^{(0)}$ lifts to $p^{\ceil{\frac{k}{d}}-1}$ many 
roots of $g$ mod $p^{\ceil{\frac{k}{d}}}$. Suppose $\sigma$ is one of the lift, 
then $\sigma+p\tau$ for any $\tau\in (\Z/p^{k-\ceil{\frac{k}{d}}}\Z)^2$ is a 
root of $(g^d \mod p^k)$. So each $\zeta^{(0)}$ lifts to 
$p^{\ceil{\frac{k}{d}}-1}\cdot p^{2(k-\ceil{\frac{k}{d}})} 
= p^{2k - \ceil{\frac{k}{d}} - 1}$ many roots of $f$ mod $p^k$. 

Now suppose $k>cd$, and let $\zeta$ be a root of $f\mod p^{cd}$ such that $\zeta^{(0)} \equiv \zeta \mod p$ is a smooth point on $g$. Consider the Taylor 
expansion of $f$ at $\zeta$:
\begin{align}
&f(\zeta+p^{cd}\bx) = [g(\zeta) + T(\bx)]^d + p^{cd}h(\zeta+p^{cd}\bx)\nonumber\\
=&\left[ g(\zeta)^d + p^{cd}h(\zeta)\right] + \sum_{l=1}^dg(\zeta)^{d-l}T(\bx)^l + \sum_{i_1+i_2\geq 1} D^{i_1,i_2}h(\zeta)p^{cd(i_1+i_2+1)}x_1^{i_1}x_2^{i_2} 
\label{eq:pert}
\end{align}
where $T(\bx):=g(\zeta+p^{cd}\bx) - g(\zeta) = \sum_{i_1+i_2\geq 1} D^{i_1,i_2}g(\zeta)p^{cd(i_1+i_2)}x_1^{i_1}x_2^{i_2}$. 
As $\zeta^{(0)}$ is a smooth point on $g$, either $D^{1,0}g(\zeta)$ or $D^{0,1}g(\zeta)$ is not zero mod $p$.  Then $s(T) = cd$, and each term in the second summand of Equality (\ref{eq:pert}) has valuation $(d-l)\ord_pg(\zeta) + lcd$.  

If $\zeta^{(0)}$ is also a point on $h\mod p$, then $\zeta$ continues to lift, and by Lemma \ref{lem:bezout}, there are at most $d^2$ many such $\zeta^{(0)}$. However, there are cases when $h(\zeta) \neq 0\mod p$, yet 
$\zeta$ continues to lift to $p^k$ for $k>cd$. 

This could only happen when $g(\zeta)^d+p^{cd}h(\zeta) \equiv 0 \mod 
p^{cd+1}$, and in which case $\ord_pg(\zeta) = c$. Now the second summand in 
Equality (\ref{eq:pert}) must have order $(d-1)c + cd$, whereas the third 
summand has order $\geq 2cd$. So now 
$s(f,\zeta) = \min\left\lbrace \ord_p\left( g(\zeta)^d+p^{cd}h(\zeta)\right), (d-1)c+cd\right\rbrace$. If $s(f,\zeta) < (d-1)c+cd$ then $\tilde{f}_{cd,\zeta} 
= \frac{f(\zeta+p^{cd}\bx)}{p^{s(f,\zeta)}}\mod p$ is a nonzero constant, and 
thus $\zeta$ does not lift. Suppose otherwise. Then
\begin{align*}
\tilde{f}_{cd,\zeta} = \frac{g(\zeta)^d+p^{cd}h(\zeta)}{p^{s(f,\zeta)}} + \frac{dg(\zeta)^{d-1}}{p^{(d-1)c}}\left( D^{1,0}g(\zeta)x_1 + D^{0,1}g(\zeta) x_2\right) \mod p, 
\end{align*}
which defines a line! By Hensel's Lemma, we are done!

So the problem boils down to determining a criterion for when 
$\ord_p\left( f(\zeta)^d+p^{cd}h(\zeta)\right)$\linebreak 
$\geq (d-1)c+cd$ and $h(\zeta)\neq 0\mod p$ happens. Also, we need to compute 
$\ord_p\left( f(\zeta)^d+p^{cd}h(\zeta)\right)$ for every lift 
$\zeta \mod p^{cd}$ for each non-isolated singular points $\zeta^{(0)}$, and 
there are exactly $p^{cd-1}$ many such $\zeta$. 

In summary, computing perturbations for each and every singular point of 
$\tf$ can be very expensive going into higher dimensions: the 
underlying singular locus might not be zero-dimensional, and thus imply  
the calclulation of a number of perturbations super-linear in $p$.

It turns out for {\em some} families of curves, non-isolated singular points 
partitioned into groups that each lift uniformly. We will pursue this 
improvement in future work. 

\section*{Acknowledgements} 
We are grateful to Daqing Wan for helpful comments on curves and 
error correcting codes. 

\bibliographystyle{alpha}
\bibliography{myrefs}
\end{document}